\documentclass[final]{amsart}

\usepackage[english]{babel}
\usepackage[utf8]{inputenc}

\usepackage[T1]{fontenc}
\usepackage{lmodern}
\usepackage[LogicalSystemsSans]{ak_logic}
\usepackage[mathb]{mathabx}
\usepackage{enumitem}
\usepackage[notref,notcite,color]{showkeys}
\usepackage{url}

\newcommand*{\Theorem}{Theorem}
\newcommand*{\Proposition}{Proposition}
\newcommand*{\Lemma}{Lemma}
\newcommand*{\Corollary}{Corollary}
\newcommand*{\Definition}{Definition}
\newcommand*{\Remark}{Remark}
\newcommand*{\Notation}{Notation}

\theoremstyle{plain}
\newtheorem{theorem}{\Theorem}
\newtheorem{proposition}[theorem]{\Proposition}
\newtheorem{corollary}[theorem]{\Corollary}
\newtheorem{lemma}[theorem]{\Lemma}
\theoremstyle{definition}
\newtheorem{definition}[theorem]{\Definition}
\theoremstyle{remark}
\newtheorem{remark}[theorem]{\Remark}

\usepackage{prettyref}
\newrefformat{thm}{\Theorem~\ref{#1}}
\newrefformat{pro}{\Proposition~\ref{#1}}
\newrefformat{cor}{\Corollary~\ref{#1}}
\newrefformat{lem}{\Lemma~\ref{#1}}
\newrefformat{rem}{\Remark~\ref{#1}}
\newrefformat{def}{\Definition~\ref{#1}}

\begin{document}

\title{From Bolzano-Weierstra{\ss} to Arzelà-Ascoli}
\author{Alexander P. Kreuzer}
\address{Fachbereich Mathematik, Technische Universit\"{a}t Darmstadt\\
Schlossgartenstra{\ss}e~7, 64289 Darmstadt, Germany
}
\email{akreuzer@mathematik.tu-darmstadt.de}
\thanks{The author is supported by the German Science Foundation (DFG Project KO 1737/5-1).}

\begin{abstract}
  We show how one can obtain solutions to the Arzelà-Ascoli theorem using suitable applications of the Bolzano-Weierstra{\ss} principle. With this, we can apply the results from \cite{aK} and obtain a classification of the strength of instances of the Arzelà-Ascoli theorem and a variant of it.

  Let \lp{AA} be the statement that each equicontinuous sequence of functions $f_n\colon [0,1] \to [0,1]$ contains a subsequence that converges uniformly with the rate $2^{-k}$ and let \lp{AA_{weak}} be the statement that each such sequence contains a subsequence which converges uniformly but possibly without any rate.

  We show that \lp{AA} is instance-wise equivalent over \ls{RCA_0} to the Bolzano-Weierstra{\ss} principle \lp{BW} and that \lp{AA_{weak}} is instance-wise equivalent over \ls{WKL_0} to \lp{BW_{weak}}, and thus to the strong cohesive principle (\lp{StCOH}).
  Moreover, we show that over \ls{RCA_0} the principles \ls{AA_{weak}}, $\lp{BW_{weak}}+\lp{WKL}$ and $\lp{StCOH}+\lp{WKL}$ are equivalent.
\end{abstract}

\maketitle

The Arzelà-Ascoli theorem is the following, well known statement:
\begin{quote}
  Let $f_n\colon [0,1] \to [0,1]$ be an equicontinuous sequence of functions. Then there exists a subsequence of $\big(f_n\big)_{n\in\Nat}$ which converges uniformly.
\end{quote}
Instead of the interval $[0,1]$ one could take any compact set.
The term \emph{equicontinuous} means that 
\[
\Forall{l}\Forall{x\in[0,1]} \Exists{j} \Forall{n} \Forall{y\in [0,1]} \left(\left|x-y\right| < 2^{-j} \IMPL \left|f_n(x) - f_n(y)\right| < 2^{-l} \right)
.\]

We will give two different formalizations of this theorem, show how these can be reduced to suitable instances of the Bolzano-Weierstra{\ss} principle and, using this, obtain a classification of them in the sense of reverse mathematics and computable analysis.

\section{Bolzano-Weierstra{\ss}}
In \cite{aK} we investigated the strength of the following two variants of the Bolzano-Weierstra{\ss} principle:
\begin{itemize}
\item The (strong) Bolzano-Weierstra{\ss} principle (\lp{BW}) is the statement that each bounded sequence of real numbers contains a subsequence converging at the rate $2^{-k}$. (This is the usual formulation in reverse mathematics. The rate $2^{-k}$ stems from the fact that real numbers are coded as sequences that converge at this rate. However, $2^{-k}$ is just an arbitrarily chosen rate. In fact, one can easily convert a sequence converging at a given rate into a sequence converging at any other given rate.)
\item The weak Bolzano-Weierstra{\ss} principle (\lp{BW_{weak}}) is the statement that each bounded sequence of real numbers contains a subsequence that converges but possibly without any rate recursive in the system.
\end{itemize}
It is well known that \lp{BW} is equivalent to \lp{ACA_0}, see \cite{sS09}. We showed that instances of \lp{BW} are equivalent to instances of \lp[\Sigma^0_1]{WKL}, that is \lp{WKL} for trees given by a $\Sigma^0_1$\nobreakdash-predicate. Moreover we showed that the principle \lp{BW_{weak}} is (instance-wise) equivalent to the so-called strong cohesive principle (\lp{StCOH}). In particular, it does not imply \lp{ACA_0}. See \cite{aK}.

We will write \lpp{BW_{(weak)}}{(x_n)} for \lp{BW_{(weak)}} restricted to the sequence $(x_n)$.

\section{Arzelà-Ascoli}

According to the two variants of the Bolzano-Weierstra{\ss} principle we can formalize the Arzelà-Ascoli theorem in two different variants.
Before we will come to this, we define equicontinuity.
\begin{definition}[equicontinuity]\label{def:equi}
  A sequence of functions $f_n\colon [0,1] \to [0,1]$ is said to be \emph{equicontinuous} if the functions $f_n$ are continuous and  have a common, continuous modulus of continuity, i.e.\ there exists a continuous function $\phi(x,l)$ satisfying
  \begin{equation}\label{eq:equi}
  \Forall{l} \Forall{n} \Forall{x,y\in [0,1]} \left(\left|x-y\right| < 2^{-\phi(x,l)} \IMPL \left|f_n(x) - f_n(y)\right| < 2^{-l} \right)
  .\end{equation}

  We call a sequence of functions \emph{uniformly equicontinuous} if the modulus of continuity does not depend on $x$, i.e.~$\phi(x,l)=\phi'(l)$.
\end{definition}
Recall that in  \cite{sS09} continuous functions are defined in a way such that the modulus of continuity is definable. Thus, \prettyref{def:equi} is just a straight forward generalization.

\begin{definition}[Arzelà-Ascoli]\label{def:aa}
   Let $f_n\colon [0,1] \to [0,1]$ be an arbitrary equicontinuous sequence of functions. 
  \begin{itemize}
  \item The (strong) variant of the Arzelà-Ascoli  theorem (\lp{AA}) is the statement that there exists a subsequence $f_{g(n)}$ which converges uniformly at the rate $2^{-k}$, i.e.
    \[
    \Forall{k} \Forall{n,n'>k} \sup_{x\in[0,1]} \left(\left|f_{g(n)}(x) - f_{g(n')}(x)\right|\right) < 2^{-k}
    .\]
  \item The weak variant of the Arzelà-Ascoli theorem (\lp{AA_{weak}}) is the statement that  there exists a subsequence $f_{g(n)}$ which converges uniformly possibly without any given rate, i.e.
    \[
    \Forall{k} \Exists{m} \Forall{n,n'>m} \sup_{x\in[0,1]} \left(\left|f_{g(n)}(x) - f_{g(n')}(x)\right|\right) < 2^{-k}.
    \]
  \end{itemize}

  If we additionally assume $(f_n)_{n\in\Nat}$ to be \emph{uniformly} equicontinuous we write \lp{AA^\textit{uni}} resp.~\lp{AA^\textit{uni}_{weak}}.
  In the case, where we restrict us the instance given by this particular $\big(f_n\big)_n$ we write \lpp{AA^{(\textit{uni})}_{(weak)}}{\big(f_n\big)_n}, resp.~\lpp{AA^{(\textit{uni})}_{(weak)}}{X} if $X$ is a code for $\big(f_n\big)$.
\end{definition}

Simpson showed that \lp{AA} is equivalent to \lp{ACA_0}, see \cite{sS09,sS84}. (Actually he did not assume the existence of a modulus of equicontinuity but only equicontinuity. However, relative to \lp{ACA_0}, as well as this formulation of the Arzelà-Ascoli theorem, the modulus can be constructed out of this.)
In \cite{uK98c} the strength of instances of \lp{AA} was investigated. It was show that instances of \lp{AA} follow from instances of \lp[\Pi^0_1]{CA} and a weak non-standard axiom.

We will now show how one can reduce (instance-wise) the principle \lp{AA^\textit{uni}} and \lp{AA^\textit{uni}_{weak}} to \lp{BW} resp.\ \lp{BW_{weak}}. Since the Arzelà-Ascoli theorem trivially implies the Bolzano-Weierstra{\ss} principle, we obtain a tight classification.
In \prettyref{sec:equi} we will deal with the non-uniformly equicontinuous case.

In the following, we will denote by $q(i)$ an enumeration of $\Rat \cap [0,1]$. By $[0,1]^{\Nat}$ we will denote the usual product space with the usual product metric $d((x_i),(y_i)) := \sum_{i\in\Nat} 2^{-i} \left| x_i - y_i \right|$.

\begin{lemma}\label{lem:1}
  Let $f_n\colon [0,1] \to [0,1]$ be a uniformly equicontinuous sequence of functions.
  Relative to $\ls{RCA_0}$ the following are equivalent:
  \begin{enumerate}[label=(\roman*)]
  \item\label{e1} $\big(f_n\big)_{n\in\Nat}$ converges uniformly,
  \item\label{e3} $\big(f_n\big)_{n\in\Nat}$ converges pointwise on $\Rat \cap [0,1]$, i.e.~$\big(f_n(q(i))\big)_i$ converges in $[0,1]^{\Nat}$ for $n\to\infty$.
  \end{enumerate}
\end{lemma}
\begin{proof}
  The implication $\ref{e1} \IMPL \ref{e3}$ is  clear. We show $\ref{e3}\IMPL \ref{e1}$.
  Suppose that $\big(f_n\big)$ converges pointwise on $\Rat \cap [0,1]$. We have to show that for a given $k$  there is an $m$ such that 
  \begin{equation}\label{eq:1}
  \Forall{n,n' > m}  \sup_{{x\in[0,1]}} \left(\left|f_{n}(x) - f_{n'}(x)\right|\right) < 2^{-k}
  .\end{equation}

  By uniform equicontinuity there is a $j:=\phi'(k+2)$ such that 
  \begin{equation}\label{eq:2}
  \Forall{n}\Forall{x,y\in [0,1]} \left( \left|x-y\right| < 2^{-j} \IMPL \left|f(x)-f(y) \right| < 2^{-(k+2)}\right)
  .\end{equation}

  Now by assumption $\big(f_n(y)\big)$ converges for each $y$ in  the set 
  \[
  Y_j:= \left\{ \frac{i}{2^{-(j+1)}} \sizeMid 0 \le i \le 2^{-(j+1)}\right\} \subseteq \Rat
  .\]
  Moreover by the definition of $d$ we know that all $\big(f_n(y)\big)$ with $y\in Y_j$ are $\epsilon$-close to their limit-points in $[0,1]$ if $\big(f_n(q(i))\big)$ is $2^{-\max q^{-1}(Y_j)} \epsilon$-close to its limit-point in $[0,1]^\Nat$.
    In particular, we get by setting $\epsilon$ to $2^{-(k+2)}$ 
  \begin{equation} \label{eq:3}
     \Exists{m'} \Forall{y\in Y_j} \Forall{n,n'>m'} \left( \left| f_n(y)-f_{n'}(y) \right| < 2^{-(k+2)}\right).
  \end{equation}

  We claim that setting $m:= m'$ satisfies \eqref{eq:1}. Indeed, for $n,n' > m$ we have
  \begin{align*}
    \left|f_{n}(x) - f_{n'}(x)\right| & < \left| f_n(y) - f_{n'}(y) \right| + 2 \cdot 2^{-(k+2)} && \parbox{4.9cm}{by \eqref{eq:2}, where $y$ is a $2^{-(j+1)}$\nobreakdash-close to $x$ element of $Y_j$} \\
    & < 3 \cdot 2^{-(k+2)} && \text{by \eqref{eq:3}} \\
    & < 2^{-k}.
  \end{align*}
  The lemma follows.
\end{proof}

For convergence with a rate we have a similar result.
\begin{corollary}\label{cor:1}
  Let $f_n\colon [0,1] \to [0,1]$ be a uniformly equicontinuous sequence of functions.
  Then relative to $\ls{RCA_0}$ the following are equivalent:
  \begin{enumerate}[label=(\roman*)]
  \item $\big(f_n\big)_{n\in\Nat}$ converges uniformly at a given rate,
  \item $\big(f_n\big)_{n\in\Nat}$ converges pointwise on $\Rat \cap [0,1]$ at a given rate.
  \end{enumerate}
  We do not state a fixed rate here since the rates may differ. However, they can be uniformly calculate from each other and the modulus of uniform equicontinuity.
\end{corollary}
\begin{proof}
  Similar to \prettyref{lem:1}. Again the implication $\ref{e1} \IMPL \ref{e3}$ is trivial. For the implication $\ref{e3} \to \ref{e1}$ note that a careful inspection of the proof of \prettyref{lem:1} yields that
  $\big(f_n\big)$ is $2^{-k}$-close to its uniform limit point if $\big( f_n(q(i)) \big)$ is $2^{-(k+\max q^{-1}(Y_{\phi'(k)}))}$-close in $[0,1]^{\Nat}$.
\end{proof}

Now to reduce the Arzelà-Ascoli theorem for a sequence of uniformly equicontinuous functions $(f_n)_{n\in\Nat}$ to a suitable instance of the Bolzano-Weierstra{\ss} principle we considered the following mapping
\[
F\colon f \mapsto {\big(f(q(i))\big)}_{i} \in [0,1]^\Nat
.\]

With this function we get a sequence $\big(F(f_n)\big)_{n\in\Nat}$ in $[0,1]^{\Nat}$ and by Lemma~\ref{lem:1} we know that for any subsequence given by $g$ we have
\begin{equation}\label{eq:con}
\big(F(f_{g(n)})\big)_{n\in\Nat} \text{ converges in } [0,1]^{\Nat}  \qquad \text{if{f}} \qquad \big(f_{g(n)}\big)_{n\in\Nat}\text{ converges uniformly}
.\end{equation}

Since one can map the unit interval $[0,1]$ isometrically into the Cantor space $2^\Nat$ (take for instance the binary expansion), we can modify $F$ such that it maps into $2^\Nat$ and the equivalence in \eqref{eq:con} remains true. Since ${(2^\Nat)}^\Nat$ is homeomorphic to the Cantor space, we can again modify $F$ and obtain a function $F'$ such that
\begin{equation}\label{eq:con2}
\big(F'(f_{g(n)})\big)_{n\in\Nat} \text{ converges in } 2^\Nat  \qquad \text{if{f}} \qquad \big(f_{g(n)}\big)_{n\in\Nat}\text{ converges uniformly}
.\end{equation}

Thus, we reduced \lp{AA_{weak}^\textit{uni}} to the weak Bolzano-Weierstra{\ss} principle on the Cantor space, which is (instance-wise) equivalent to \lp{BW_{weak}}, see \cite[Lemma~2.1]{aK}.  Hence, we obtain the following theorem.

\begin{theorem}\label{thm:aau}
  Over \ls{RCA_0} the principles \lp{AA_{weak}^\textit{uni}} and \lp{BW_{weak}} are instance-wise equivalent, i.e.\ there exists codes of Turing machines $e_1$, $e_2$ such that
  \begin{enumerate}[label=\arabic*)]
  \item\label{enum:aau1} $\lp{RCA_0} \vdash \Forall{X} \left(\lpp{BW_{weak}}{\{e_1\}^X} \IMPL \lpp{AA_{weak}^\textit{uni}}{X}\right)$,
  \item\label{enum:aau2} $\lp{RCA_0} \vdash \Forall{X} \left(\lpp{AA_{weak}^\textit{uni}}{\{e_2\}^X} \IMPL \lpp{BW_{weak}}{X}\right)$.
  \end{enumerate}
  In particular, \lp{AA_{weak}^\textit{uni}} is also instance-wise equivalent to the strong cohesive principle.
\end{theorem}
\begin{proof}
   One checks that the argument in the discussion before the theorem formalizes in \ls{RCA_0}. Setting $e_1$ to be the code of the Turing machine which calculates $F'$ yields then \ref{enum:aau1}.

   For \ref{enum:aau2} let $e_2$ be the code of the Turing machine which maps the sequence of numbers coded by $X$ to the sequence of constant functions having that values. This instance of the Arzelà-Ascoli theorem trivially implies the weak Bolzano-Weierstra{\ss} principle.

   For the equivalence to the strong cohesive principle see \cite[Theorem~3.2]{aK}.
\end{proof}

Replacing \prettyref{lem:1} by \prettyref{cor:1} and \lp{BW_{weak}} by \lp{BW} and noting that the rate of convergence in the proof can be explicitly calculate yields the following corollary.
\begin{corollary}\label{cor:aau}
  Over \ls{RCA_0} the principles \lp{AA^\textit{uni}} and \lp{BW} are instance-wise equivalent. In particular, \lp{AA^\text{uni}} is also instance-wise equivalent to \lp[\Sigma^0_1]{WKL}.
\end{corollary}

\section{Uniform equicontinuity versus equicontinuity}\label{sec:equi}

\begin{proposition}\label{pro:equi}
  Let $\big(f_n\big)$ be an equicontinuous sequence of functions. The system \ls{WKL_0} proves that $\big(f_n\big)$ is uniformly equicontinuous.
\end{proposition}
\begin{proof}
  Let $\phi(x,l)$ be a modulus of equicontinuity for $\big(f_n\big)$. By Theorem IV.2.2 in \cite{sS09} for each $l$ there exists the maximum of $\lambda x . \phi(x,l)$. A careful inspection of the proof of this theorem shows that this process parallelizes. Thus. we can define a function $\phi'(l)$ such that $\phi(x,l) \le \phi'(l)$ for all $x\in [0,1]$, $l\in \Nat$. Since \eqref{eq:equi} in \prettyref{def:equi} is monotone in $\phi$, the function $\lambda x,l . \phi'(l)$ satisfies the sentence. Thus, it is a modulus of uniform equicontinuity.
\end{proof}

Using this, we can immediately refine \prettyref{thm:aau} and \prettyref{cor:aau} and obtain the following corollary and theorem.
\begin{corollary}\label{cor:weak}
  Over \ls{WKL_0} the principles \lp{AA_{weak}}, \lp{BW_{weak}}, \lp{StCOH} are instance-wise equivalent.
  (Actually only to show \lp{AA_{weak}} the principle \lp{WKL} is needed.)
\end{corollary}
\begin{proof}
  By \prettyref{pro:equi} every equicontinuous sequence of functions is in \lp{WKL_0} uniformly equicontinuous. Thus, there is no difference between \lp{AA_{weak}^\textit{uni}} and \lp{AA_{weak}} in \lp{WKL_0} and this corollary follows from \prettyref{thm:aau}.
\end{proof}

\begin{theorem}\label{thm:aa}
  Over \ls{RCA_0} the principles \lp{AA}, \lp{BW}, \lp[\Sigma^0_1]{WKL} are instance-wise equivalent.
\end{theorem}
\begin{proof}
  We show that there exists an $e$ such that
  \[
  \ls{RCA_0} \vdash \Forall{X} \left(\lpp{BW}{\{e\}^{X}} \IMPL \lpp{AA}{X}\right) 
  .\]

  Fix an $X$ that codes a sequence of equicontinuous functions.
  Note that in the proof of \prettyref{pro:equi} the principle \lp{WKL} is only used for trees recursive in $\phi(x,l)$ and thus recursive in $X$. 

  To ask whether a given node $x$ in a 0/1-tree $T$ has infinitely many successors is the $\Pi^0_1$-statement $\Forall{n} \Exists{y\in 2^n} x\ast y \in T$. 
  If we can decide this for each node in an infinite 0/1\nobreakdash-tree $T$, we can build an infinite branch by searching for the leftmost branch of nodes having infinitely many successors.
  Thus, an instance of $\Pi^0_1$\nobreakdash-comprehension recursive in $X$ suffices to show \lp{WKL}.
  By Theorem~5.5 and Lemma~4.1 of \cite{uK00} a suitable instance of \lp{BW} (even the weaker principle of convergence for monotone sequences) implies this instance of $\Pi^0_1$-comprehension.
  Thus, we can find an $e'$ such that \lpp{BW}{\{e'\}^{X}} implies that the modulus of uniform continuity $\phi'$ for the sequence of functions coded by $X$ exists.
  By the proof of \prettyref{cor:aau} there is an $e''$ such that \lpp{BW}{\{e''\}^{X}} shows that the sequence of functions converges pointwise on $\Rat\cap[0,1]$.

  In the argument of the proof of \prettyref{cor:aau} the modulus $\phi'$ is only used after the pointwise converging sequence is built. Thus, this is enough to show \lpp{AA}{(f_n)}.

  Now the two instances of \lp{BW} can be coded into an instance of the Bolzano-Weierstra{\ss} principle on $[0,1]^2$. This instance  is again equivalent to an instance of \lp{BW}. Let $e$ be the code of a Turing-machine which computes this instance. This concludes the proof.
\end{proof}

This yields as corollary the following classification of Simpson of the principle \lp{AA}.
\begin{corollary}[\cite{sS09,sS84}]
  Over \lp{RCA_0} the principles \lp{AA} and \lp{ACA_0} are equivalent.
\end{corollary}
We also obtain the following computational classification.
\begin{corollary}
  Let $d$ be a Turing degree with $d \gg 0'$, i.e.~$d$ contains an infinite branch for each infinite $0'$-computable 0/1-tree.
  Then each computable sequence of equicontinuous functions $f_n \colon [0,1] \to [0,1]$ has a subsequence $\big(f_{(g(n))}\big)$ computable in $d$, which converges uniformly with the rate $2^{-n}$.
\end{corollary}
\begin{proof}
  A degree $d\gg 0'$ contains solutions to each computable instance of \lp[\Sigma^0_1]{WKL}. The corollary follows from this and \prettyref{thm:aa}.
\end{proof}
\noindent
Since \lp{AA} instance-wise implies \lp[\Sigma^0_1]{WKL}, this corollary is optimal.

We will now show that \lp{WKL_0} is necessary in \prettyref{cor:weak} by showing that \lp{AA_{weak}} implies it.
Since \lp{BW_{weak}} is equivalent to $\lp{StCOH}$ which does not imply \lp{WKL}, see \cite[Lemma~9.14]{CJS01} and note that in $\omega$-models \lp{COH} and \lp{StCOH} are the same, the system cannot be weakened to \ls{RCA_0}.

\begin{proposition}\label{pro:wkl}
  \[
  \ls{RCA_0} \vdash \lp{AA_{weak}} \IMPL \lp{WKL}
  \]
\end{proposition}
\begin{proof}
  We will show that $\NOT \lp{WKL} \IMPL \NOT \lp{AA_{weak}}$. The construction is inspired by Theorem IV.2.3 of \cite{sS09}.

  Let $T\subseteq 2^\Nat$ be a tree which is infinite but does not have an infinite path. Such a tree exists by $\NOT \lp{WKL}$.
  Define $\tilde{T}$ to be set of $u\in 2^{<\Nat}$ such that $u \notin T \AND \Forall{t \sqsubsetneq u} \left( t\in T \right)$.

  Let $C$ be the Cantor middle-third set given by
  \[
  C:=\left\{ \sum_{i\in\Nat} \frac{2 \cdot f(i)}{3^{i+1}} \in  [0,1] \sizeMid  f\in 2^\Nat \right\}
  .\]
  Further, for each $s\in 2^{<\Nat}$ let
  \[
  a_s := \sum_{i<\lth(s)} \frac{2\cdot(s)_{i}}{3^{i+1}}, \qquad b_s = a_s + \frac{1}{3^{\lth(s)}}
  .\]
  We now consider the set $S := \bigcup_{s\in \tilde{T}} [a_s,b_s]$. Since the intervals $[a_s,b_s]$ with $s \in \tilde{T}$ are disjoint, for each $x\in S$ there is exactly one $s$ such that $x\in [a_s,b_s]$.

  We claim that $C\subseteq S$. Indeed for each $x\in C$ the exists a unique $f$ such that $x \in [a_{f(n)}, b_{f(n)}]$. 
  Since the tree $T$ has no infinite path and, therefore, $f$ is no such path, there is an $n$ such that  $f(n)\in \tilde{T}$ and $x\in [a_{f(n)},b_{f(n)}]$.

  For each $x\notin S$ we have $x\notin C$. By the properties of $C$ there exists a unique $s$, such that $x\in (b_{s \ast \langle 0 \rangle}, a_{s \ast \langle 1 \rangle})$. Since $a_{s \ast \langle 1 \rangle}, b_{s \ast \langle 0 \rangle} \in C\subseteq S$ there are unique $v,w\in \tilde{T}$, such that 
  \[
  b_{s \ast \langle 0 \rangle} \in [a_v,b_v], \qquad
  a_{s \ast \langle 1 \rangle} \in [a_w,b_w]
  \]
  and thus $x \in (b_v,a_w)$ and $(b_v,a_w) \cap S = \emptyset$ for unique $v,w \in \tilde{T}$.

  We now construct an equicontinuous sequence of functions $f_n\colon [0,1] \to [0,1]$ such that $f_n$ does converge pointwise to the constant $0$ function but does not converge uniformly.

  We define $f_n$ on the set $S$ and use linear interpolation on $[0,1]\setminus S$.

  We set $f_n$ to $1$ on $[a_s,b_s]$ if $\lth(s) > n$ and to $0$ if $\lth(s) \le n$. It follows that  $f_n$ converges pointwise to the constant $0$ function.
  Since $T$ is infinite, there are arbitrary long $s$ and we can find for each $n$ an $x$ such that $f_n(x)=1$.
  Thus, $f_n$ does not converge uniformly.

  In total $f_n$ is given by the following expression.
  \[
  f_n(x)\! :=\!
  \begin{cases}
    1 &\!\! \text{if $x\in [a_s,b_s]$ for $s\in\tilde{T}$ and $\lth(s)\! >\! n$,} \\
    0 &\!\! \text{if $x\in [a_s,b_s]$ for $s\in\tilde{T}$ and $\lth(s)\! \le\! n$,} \\
    f_n(b_v) + \frac{x-b_v}{a_w-b_v}(f_n(a_w)-f_n(b_v)) & 
    \!\!\parbox{4cm}{if $x\in (b_v,a_w)$ for $v,w\in \tilde{T}$ \\\hspace*{1em} and $(b_v,a_w) \cap S=\emptyset$.} 
  \end{cases}
  \]
  The functions $f_n$ are well-defined since $s$ resp.~$u,v$ are uniquely determined for each $x$. It is clear the each $f_n$ is continuous. We argue now that the sequence is equicontinuous.
  On the intervals $(a_s,b_s)$ with $s\in \tilde{T}$ each $f_n$ is constant and, therefore, we can define the modulus of equicontinuity here easily.
  On the intervals $[b_v,a_w]$ with $v,w\in \tilde{T}$ and $(b_v,a_w) \cap S=\emptyset$ the gradient of each $f_n$ is at most $\left|\frac{1}{a_w-b_v}\right|$. Thus, the modulus of equicontinuity can also be defined here.

  The sequence of functions $(f_n)$ provides a counterexample to \lp{AA_{weak}}. This concludes the proof of the proposition.
\end{proof}

With this we obtain the following classification of \lp{AA_{weak}}.
\begin{theorem}\label{thm:aaweq}
  Relative to \lp{RCA_0} the following are equivalent:
  \begin{enumerate}[label=(\roman*)]
  \item \label{ea1} $\lp{AA_{weak}}$,
  \item \label{ea2} $\lp{BW_{weak}} + \lp{WKL}$,
  \item \label{ea3} $\lp{StCOH}+\lp{WKL}$.
  \end{enumerate}
\end{theorem}
\begin{proof}
  The implication $\ref{ea1} \IMPL \ref{ea2}$ follows from \prettyref{thm:aau} and \prettyref{pro:wkl}; the implication $\ref{ea2} \IMPL \ref{ea1}$ follows from  \prettyref{thm:aau} and \prettyref{pro:equi}. For $\ref{ea2} \IFF \ref{ea3}$ see \cite[Theorem~3.2]{aK}.
\end{proof}

In the case of \lp{AA_{weak}} the principle \lp{WKL} is only needed in the verification of the solution and not in the computation of it. Thus, we obtain the following corollary.
\begin{corollary}\label{cor:aawcomp}
  Every equicontinuous sequence $f_n\colon [0,1] \to [0,1]$ contains a $low_2$ uniformly converging subsequence $\big( f_{g(n)}\big)_n$ (possibly without a computable rate).
\end{corollary}
\begin{proof}
  By \cite[Theorem 3.5]{aK} each computable instance of \lp{BW_{weak}} has a $low_2$ solution. Since $g$ is computed by an instance of \lp{BW_{weak}}, see discussion before \prettyref{thm:aau}, it is also $low_2$. 

  Note that for this computation the modulus of uniform equicontinuity is not needed.
  The only use of the modulus of uniform equicontinuity is the verification in \prettyref{lem:1}. Thus, it suffices that the modulus of uniform equicontinuity exists and we do not have to compute it.
\end{proof}

Using \prettyref{thm:aaweq} one can extend the conservation and program extraction results obtained in \cite{CJS01,CSYta} and \cite{swkl} to \lp{AA_{weak}} and obtain the following theorem.
\begin{theorem}\mbox{}
  \begin{enumerate}[label=\arabic*)]
  \item\label{enum:r1} \lp{AA_{weak}} is $\Pi^1_1$\nobreakdash-conservative over $\ls{RCA_0}+\lp[\Pi^0_1]{CP}$ and $\ls{RCA_0}+\lp[\Sigma^0_2]{IA}$.
  \item\label{enum:r2} From a proof of a sentence of the form $\Forall{f\in \Nat^\Nat}\Exists{x\in \Nat} A_\qf(f,x)$ in the system $\lp{WKL_0^\omega}+ \lp[\Pi^0_1]{CP} +\lp{AA_{weak}}$ one can extract a primitive recursive term $t$ realizing $x$, i.e.~a term $t$ such that $\Forall{f} A_\qf(f,t(f))$ holds.
  \end{enumerate}
  In particular, \lp{AA_{weak}} is $\Pi^0_2$-conservative over \ls{PRA}.
\end{theorem}
\begin{proof}
  For \ref{enum:r1} see \cite{CJS01} for the conservativity over $\ls{RCA_0}+\lp[\Sigma^0_2]{IA}$ and \cite{CSYta} for the conservativity over $\ls{RCA_0}+\lp[\Pi^0_1]{CP}$ and note that \lp{StCOH} is equivalent to $\lp{COH} + \lp[\Pi^0_1]{CP}$.

  For \ref{enum:r2} see \cite[Corollary~38]{swkl}.
\end{proof}

\begin{remark}
  The classification of the Arzelà-Ascoli theorem can also be formulated in terms of the Weihrauch-lattice. 
  We will not introduce the notation for the Weihrauch-lattice but refer the reader to \cite{BG11,BGM12}.

  Continuous functions $f\colon [0,1]\to [0,1]$ can be represented as associates on the Baire-space, see \cite{sK59b,gK59}. We will denote this space by $\mathcal{C}([0,1],[0,1])$ and the representation by $\delta_{\mathcal{C}}$. With this we can formulate the Arzelà-Ascoli theorem as partial multifunction between realized spaces which maps sequences of equicontinuous functions to its uniform limit points, i.e.
  \[
  \mathsf{AA} :\subseteq \left(\mathcal{C}([0,1],[0,1]), \delta_{\mathcal{C}}\right)^\Nat \rightrightarrows \left(\mathcal{C}([0,1],[0,1]),\delta_{\mathcal{C}}\right)
  \]
  with $dom(\mathsf{AA})= \{\,(f_n) \mid (f_n) \text{ equicontinuous}\,\}$. As customary in the Weihrauch-lattice, we do not assume that any extra information---like a modulus of equi\-con\-ti\-nu\-i\-ty---is given as input. However, we are working in full set-theory; thus we know that a modulus of uniform equicontinuity exists.

  The weak variant of the Arzelà-Ascoli theorem can be modelled by using the derived representation $\delta'_{\mathcal{C}}$. Here an element is represented by a sequence converging to a name in the original representation. In other words, $\delta'_{\mathcal{C}} := \delta_{\mathcal{C}} \circ \mathrm{lim}$. See \cite[Sec.~5]{BGM12} and \cite{mZ07}. Thus
  \[
  \mathsf{AA_{weak}} :\subseteq \left(\mathcal{C}([0,1],[0,1]), \delta_{\mathcal{C}}\right)^\Nat \rightrightarrows \left(\mathcal{C}([0,1],[0,1]),\delta'_{\mathcal{C}}\right)
  .\]  
  
  It is easy to see that $\mathrm{lim}$ is sufficient to build a modulus of uniform continuity.
  Thus, we can argue as in the proof of \prettyref{thm:aa} and obtain
  \[
  \mathsf{AA} \le_{W} \mathrm{lim} \times \mathsf{BWT_{[0,1]^\Nat}} \equiv_{W} \mathsf{BWT}_{\Real}
  \]
  and in total $\mathsf{AA} \equiv_{W} \mathsf{BWT}_{\Real}$.

  Similarly, we obtain that 
  \[
  \mathsf{AA_{weak}} \equiv_{W} \mathsf{WBWT_{[0,1]^\Nat}} \equiv_{W} \mathsf{WBWT}_{\Real},
  \]
  where $\mathsf{WBWT}_X$ is the weak Bolzano-Weierstra{\ss} principle. Here the calculation of the modulus of uniform equicontinuity can be done in the representation, since it involves a $\mathrm{lim}$, and $\mathsf{WBWT_{[0,1]^\Nat}}\equiv_W \mathsf{WBWT}_{\Real}$ gives the values of solution on the rational numbers in $[0,1]$. With this an associate of solution can be defined.
\end{remark}

\bibliographystyle{amsplain}
\bibliography{swkl,weak,primrec}
\end{document}